\documentclass[11pt]{amsart}
\usepackage{amsmath,amsthm,amssymb}
  \usepackage[notcite, notref]{showkeys}
\usepackage{amssymb,mathrsfs,graphicx}

\usepackage[all]{xy}
\usepackage[english]{babel}

\usepackage{amsmath,amssymb}
\usepackage{enumerate}
\usepackage{amsthm}

\usepackage{amscd}

\usepackage{amsfonts}

\usepackage{mathrsfs}
\usepackage{epsfig}

\usepackage{stmaryrd}

\usepackage{layout}

\usepackage[usenames,dvipsnames]{color}

\newtheorem{Thm}{Theorem}[section]

\newtheorem{Lem}[Thm]{Lemma}

\newtheorem{Prop}[Thm]{Proposition}

   \newtheorem{prop}[Thm] {Proposition}     
   \newtheorem{lemma} [Thm]{Lemma}

\usepackage{hyperref}

\theoremstyle{definition}
\newtheorem{defi}[Thm]{Definition}

     \newtheorem{example}[Thm]{Example}

\newtheorem{rem}[Thm]{Remark}
\newtheorem{remark}[Thm]{Remark}

\newcommand{\RR}{{\mathbb{R}}}
\newcommand{\ZZ}{{\mathbb{Z}}}

\newcommand{\cC}{{\mathcal C}}
\newcommand{\cD}{{\mathcal D}}

\newcommand{\cO}{{\mathcal O}}

\newcommand{\cS}{{\mathcal S}}

\newcommand{\cX}{{\mathcal X}}

 \newcommand{\mdeg}{{\underline{\deg}}}

\def\<{\langle}
\def\>{\rangle}

\newcommand{\oP}{\overline{P}}

\newcommand{\Spec}{\operatorname{Spec}}

\newcommand{\Pic}{{\operatorname{{Pic}}}}

\newcommand{\Div}{{\operatorname{Div}}}
\newcommand{\dv}{{\operatorname{div}}}
\newcommand{\Prin}{{\operatorname{Prin}}}

\newcommand{\Aut}{{\operatorname{Aut}}}

\newcommand{\trop}{{\operatorname{trop}}}

\newcommand{\an}{{\operatorname{an}}}
\newcommand{\val}{{\operatorname{val}}}

\newcommand{\ord}{\operatorname{ord}}

\newcommand{\double}{\genfrac..{0pt}1
{\raise -2pt\hbox{$\scriptstyle\longrightarrow$}}{\raise 4pt\hbox
{$\scriptstyle\longrightarrow$}}} 
 
\newcommand{\ssm}{\smallsetminus}

 \newcommand{\la}{\longrightarrow}
\newcommand{\ha}{\hookrightarrow}

\newcommand{\ov}{\overline}

\newcommand{\mc}{\mathcal}

 \newcommand{\Mgbst}{{\ov{\mathcal{M}}_g}}
  \newcommand{\Sgbst}{{\ov{\mathcal{S}}_g}}

\newcommand{\Mgt}{{M_{g}^{\rm trop}}}

\newcommand{\Sgt}{{S_{g}^{\rm trop}}}

 \newcommand{\Sg}{{\mathcal {G}}_g}

\newcommand{\mb}{\mathbb}
\newcommand{\ol}{\overline}
\newcommand{\wh}{\widehat}

\newcommand{\TC}{\Gamma=(G,w,\ell)}

\newcommand{\col}{\colon}
\newcommand{\ra}{\rightarrow}

\newcommand{\RtG}{R_{\Gamma}^\trop}

\newcommand{\TtG}{{T_{\Gamma}^{\rm trop}}}
\newcommand{\Ttg}{{T_g^{\rm trop}}}
 
\newcommand{\tD}{\tau_{\Div}} 
\newcommand{\tP}{\tau_{\Pic}}
\newcommand{\tS}{\tau_{\cS}}
\newcommand{\nS}{\alpha}
 
 \newcommand{\cSK}{\cS_{X_K}}

\begin{document}

\setcounter{tocdepth}{1}

 \title{Theta-characteristics on tropical curves}
 
\author[]{Lucia Caporaso, Margarida Melo, and Marco Pacini}

\address[Caporaso]{Dipartimento di Matematica e Fisica\\ Universit\`{a} Roma Tre \\ Largo San Leonardo Murialdo \\I-00146 Roma\\  Italy }\email{caporaso@mat.uniroma3.it}
 \address[Melo]{Dipartimento di Matematica e Fisica\\ Universit\`{a} Roma Tre \\ Largo San Leonardo Murialdo \\I-00146 Roma\\  Italy}
 \email{melo@mat.uniroma3.it}
\address[Pacini]{Instituto de Matem\'atica, Universidade Federal Fluminense  \\ Campus do Gragoat\'a \\ 24.210-201 Niter\'oi, Rio de Janeiro, Brazil}\email{pacini.uff@gmail.com, pacini@impa.br}

\begin{abstract}
We give an explicit description of   theta-characteristics on tropical curves  and characterize the   effective ones. We    construct the moduli space, $T_g^\trop$, for tropical theta-characteristics of genus $g$ as a generalized cone complex. We describe  the fibers of the specialization map from the moduli space $\mc S_g$, of theta-characteristics on algebraic curves,     to $T_g^\trop$.
\end{abstract}



\thanks{The first and second authors were supported by funds from MIUR via the Excellence Department Project awarded to the Department of Mathematics and Physics of Roma Tre and by the project  PRIN2017SSNZAW: Advances in Moduli Theory and Birational Classification. The second author is a member of the Centre for Mathematics of the University of
Coimbra -- UIDB/00324/2020), funded by the Portuguese Government through FCT/MCTES.
The third author was supported by CNPq-PQ, 301671/2019-2.}

\maketitle

 \noindent MSC (2020): 14H10, 14H40, 14T20.

 \noindent Keywords: Moduli space,  theta-characteristic, tropical curve, algebraic curve.

\section{Introduction}

Let $X$ be a smooth algebraic curve of genus $g$ over an algebraically closed field of characteristic different from $2$. 
A theta-characteristic on $X$ is a line bundle of degree $g-1$ on $X$ such that $L^{\otimes 2}\cong K_X$, where $K_X$ is the canonical line bundle of $X$. Theta-characteristics on algebraic curves have been extensively studied over the years because of their remarkable geometry.  In particular, there is a notion of parity on theta-characteristics (given by the parity of the dimension of its space of global sections) which is  deformation invariant (see \cite{mumford}). 
Each algebraic curve admits  $2^{2g}$ theta-characteristics, $(2^g+1)2^{g-1}$ of which are even and $(2^g-1)2^{g-1}$ of which are odd.
The moduli space $\mc S_g$ of theta-characteristics on curves of genus $g$ splits in two connected components, according to the parity.

Let now $\Gamma$ be a tropical curve of genus $g$ and let $K_\Gamma$ be the canonical divisor on $\Gamma$. 
A tropical theta-characteristic on $\Gamma$ is a divisor class $[T]\in\Pic^{g-1}(\Gamma)$ such that $[2T]=[K_\Gamma]$.
Tropical theta-characteristics on pure tropical curves have been studied by Zharkov in \cite{zharkov} and shown to be associated to orientations given by the choice of a cyclic subcurve of $\Gamma$. There are therefore $2^g$ theta-characteristics on a pure tropical curve of genus $g$, and among these exactly one is non-effective.

In this paper, we study tropical theta-characteristics on any tropical curve (not necessarily pure). In Proposition~\ref{DPS} we generalize Zharkov's results by exhibiting explicit representatives for tropical theta-characteristics obtained from orientations associated to suitable flows on the curve and classify the ones that are effective  (see Theorem \ref{thm:rank}).
Our choice of representatives is different from Zharkov's, and it is inspired by the geometry of the moduli space of stable spin curves $\Sgbst$ defined by Cornalba in \cite{C89}.
Our representatives are well behaved with respect to specialization, and we are able to   construct, in Theorem \ref{Tgthm}, a moduli space, $T_g^\trop$, for tropical theta-characteristics   as a  generalized cone complexes.

  Unlike the situation for classical algebraic curves, there is no natural notion of parity on tropical theta-characteristics.
 In our paper \cite{CMP1}, we introduced the notion of tropical spin curve, which is a tropical curve together with a spin structure, i.e., a theta-characteristic together with a suitable sign function, which  encodes the notion of parity. Theta-characteristics on tropical curves are obtained from spin structures simply by  forgetting the sign function.
In loc. cit, we constructed a moduli space $S_g^\trop$ for tropical spin curves of given genus and we showed that this space is related with Cornalba's   moduli space    via Berkovich analytification. 
There is a natural   morphism of generalized cone complexes, $S_g^\trop\la M_g^\trop$,   to the moduli space of tropical curves, $ M_g^\trop$. From the results of Section 4 we obtain    a natural factorization  $$S_g^\trop\la T_g^\trop\la M_g^\trop$$
via morphisms of generalized cone complexes.

Let now $X_K$ be a smooth algebraic curve of genus $g$ over a non-Archimedian algebraically closed   field  $K$ of characteristic different from $2$, and let $\Gamma$ be the genus $g$ tropical curve which is the skeleton of $X_K$. There is  a specialization map  $\tau\colon \Pic(X_K)\ra\Pic(\Gamma)$, which is well-known to be  a homomorphism.
 Let $S_{X_K}$ be the set of theta-characteristics on $X_K$  and let $T_\Gamma^\trop$ be the set of (tropical) theta-characteristics on $\Gamma$. Since $\tau$ is a homomorphism, it restricts to a map  $S_{X_K}\to T_\Gamma^\trop$.
 It is   quite interesting to understand  the fibers of the specialization map over a given tropical theta-characteristic of $\Gamma$. This problem has been considered by various authors: in the case of tropical plane quartics by  Baker, Len, Morrison, Pflueger and Ren  in \cite{BLMPR}, Chan, Jiradilok in \cite{CJ} and Len, Markwig in \cite{LM20}, for hyperelliptic curves by Panizzut in \cite{panizzut}
and by Jensen and Len in \cite{JL} in the case of pure tropical curves.
  Jensen and Len showed that the specialization map $S_{X_K}\to  T_{\Gamma}^\trop$ is surjective. Moreover, there are $2^g$ even theta-characteristics specializing to the non-effective tropical theta-characteristic, and each effective theta-characteristic is the image of $2^{g-1}$ even theta-characteristics and $2^{g-1}$ odd theta-characteristics.
Jensen and Len's result follows by carefully analyzing the Weil pairing on the set of $2$-torsion points in the Jacobian of smooth curves.

In the second part of this paper we propose a different approach to study the specialization map for theta-characteristics, based on the description of tropical theta-characteristics given in the first part. Our results rely on our approach to tropical theta-characteristics in \cite{CMP1}, which allows us to give an interpretation to the specialization map  using the combinatorial description of the boundary of Cornalba's moduli space $\overline{\mathcal S_g}$.
In Theorem \ref{thm:main2} we use this approach to obtain a generalization of Jensen and Len's results for any tropical curve (not necessarily pure).   We point out that by using the results in \cite{CMP1} the proof of Theorem \ref{thm:main2} is quite simple compared to the previous approaches.

\section{Preliminaries}

\subsection{Graphs} 
In this paper we denote by  $G=(V,E)$   a connected graph. 
Given a vertex $v\in V$, we write  $\deg_G(v)$, sometimes $\deg(v)$, for the number of edges of $G$ incident to $v$, with loops counting twice. Sometimes we will use the notation $E(G)$ and $V(G)$ for the set of edges and vertices of $G$.

Given a subset $R\subset E$, we will frequently abuse the notation by writing $R$ to denote the subgraph of $G$ generated by $R$. For $R\subset E$, the \emph{$R$-subdivision} (or simply \emph{subdivision}) of $G$, written $\wh G_R$, is the graph obtained from $G$ by inserting exactly one vertex in the interior of each edge $e\in R$.

An \emph{orientation} $O$ on $G$ is a pair $O=(\sigma,\tau)$ where $\sigma,\tau\col E \ra V $, called {\it source} and {\it target} map, are such that $\{\sigma(e),\tau(e)\}=\{u_e,v_e\}$, where we write $ u_e,v_e $ for the ends of $e$. 
  A \emph{sub-orientation} on $G$ is an orientation on some subdivision of $G$.

Given an orientation $O$ on $G$,   if $\wh G$ is the $R$-subdivision of $G$ for $R\subset E$,  we have a natural orientation induced on $\wh G$ (i.e. a sub-orientation on $G$), which we   call, again, $O$. 
More precisely, let $e_0$, $e_1$ be the edges of $\wh G$ obtained after inserting one vertex  in the interior of $e\in R$.
If $e_0$ is incident to $\sigma(e)$ and $e_1$ to $\tau(e)$, 
then we set $\sigma(e_0)=\sigma(e)$ and $\tau(e_1)=\tau(e)$, as in the picture.
  
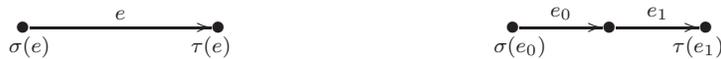
\begin{figure}[h]
\begin{equation*}
\xymatrix@=.5pc{
 &&&&*{\bullet}\ar[rrrrrr]^{e}_<{\sigma(e)} _>{\tau(e)} &&&&&&*{\bullet} 
  &&&&&&&&&*{\bullet}\ar[rrr]^{e_0}_<{\sigma(e_0)}&&&*{\bullet} \ar[rrr]^{e_1} _>{\tau(e_1)}  &&&*{\bullet}   &&&&\\
}
\end{equation*}
\end{figure}

A \emph{(vertex weighted)} graph $(G,w)$ is a (connected) graph $G$ endowed with a function $w\col V(G)\ra \mathbb Z_{\ge0}$. 
The \emph{genus} of $(G,w)$ is 
\[
g=\sum_{v\in V} w(v)+b_1(G)=\sum_{v\in V} w(v)+|E|-|V|+1.
\]
 We say that $(G,w)$ is \emph{stable} if $2w(v)-2+\deg_G(v)>0$, for every $v\in V$.  

We denote by $\mc E_G$, respectively  $\mc V_G$, the vector space over $\mb F_2$ spanned by $E$, respectively  by $V$.  We consider the linear map $\partial\col \mc E_G\ra \mc V_G$ such that for every subset $S\subset E$,
\[
\partial\left(\sum_{e\in S} e\right)=
\sum_{e\in S} (u_e+v_e).
\]
We denote by $\cC_G$ the kernel of $\partial$, and we call an element of $\cC_G$ \emph{cyclic} or  a \emph{cyclic subgraph} of $G$. 
It is well known that $\cC_G$ is generated by the {\it cycles} of $G$, (i.e. by the connected subgraphs   all of whose vertices have degree $2$.)

\subsection{Tropical curves} 

A (connected) \emph{tropical curve}  is a triple $\Gamma=(G,w,\ell)$, where $(G,w)$ is a stable   graph and $\ell$   a length function, $\ell \col E \ra \mathbb R_{>0}$.
We call $(G,w)$ the combinatorial type of $\Gamma$. The \emph{genus} of    $\Gamma=(G,w,\ell)$ is the genus of  $ (G,w)$, usually denoted by $g$.

We   view $\Gamma$ as a metric space   as follows.
Every  edge $e$ is identified with a closed interval   of length $\ell(e)$ if $e$ is not a loop,
and with
a circle of length $\ell(e)$ if $e$ is a loop.
Then every path in $\Gamma$ has a well defined length.
Now given two points $p,q\in \Gamma$ we define the distance between them, $d(p,q)$,  as the shortest length of a path in $\Gamma$
from $p$ to $q$. 

For any edge $e$ of $G$, we let $p_e\in \Gamma$ be the mid-point of $e$.

We extend the weight function   to all points of $\Gamma$ as follows
 $$
w_\Gamma(p)=
\begin{cases}
\begin{array}{ll}
w(p)\ & \text{ if } p\in V \\
0, &  \text{otherwise}.
\end{array}
\end{cases}
$$
Similarly, we set
 $$
\deg_\Gamma(p)=
\begin{cases}
\begin{array}{ll}
\deg_G(p)\ & \text{ if } p\in V \\
2, &  \text{otherwise}.
\end{array}
\end{cases}
$$
We  write $w=w_{\Gamma}$ and $\deg =\deg_\Gamma$ when no confusion is likely.

A \emph{divisor} on $\Gamma$ is a formal sum $D=\sum_{p\in \Gamma} D(p)p$, with $D(p)\in\mathbb Z$, where $D(p)$ is nonzero only for a finite set of points of $\Gamma$. We write $D\geq 0$ if $D(p)\geq 0$ for all $p\in \Gamma$, and we say $D$ is {\emph  {effective}}.
The degree of $D$ is the integer $\sum_{p\in \Gamma} D(p)$. We let $\Div(\Gamma)$ be the group of divisors on $\Gamma$ and $\Div^d(\Gamma)\subset \Div(\Gamma)$ the subset of divisors of degree $d$.  

A \emph{rational function} $f$ on $\Gamma$ is a continuous piecewise linear function $f\col\Gamma\ra \mathbb R$ with integral slopes and finitely many pieces. 
The points of $\Gamma$ where $f$ is not linear will be called {\it critical} points of $f$.
Given a rational function $f$ on $\Gamma$, for every $p\in \Gamma$, we let 
$\ord_p(f)$ be the sum of the slopes of $f$ outgoing $p$ (there are $\deg_{\Gamma}(p)$ of them). If $f$ is linear at $p$  then $\ord_p(f)=0$, therefore  the following
 
\[
\dv(f):=\sum_{p\in \Gamma} \ord_p(f)p  
\]
is a finite sum,   hence a divisor, and it is easy to see that   $\dv(f)\in \Div^0(\Gamma)$. 
The divisors of the form $\dv(f)$ are called
  \emph{principal divisors}. The set of all principal divisors is a subgroup of $\Div(\Gamma)$ denoted by
 $\Prin(\Gamma)$.
 The \emph{Picard group} of $\Gamma$ is  
\[
\Pic(\Gamma):=\Div(\Gamma)/\Prin(\Gamma),
\]
and $\Pic^d(\Gamma)\subset \Pic(\Gamma)$ is the  set of divisor classes  of degree $d$. For   $D\in \Div(\Gamma)$  we denote by $[D]\in\Pic(\Gamma)$ its class and we write $D\sim D'$ if $[D]=[D']$. 

\begin{example}
\label{fe}
 Fix a non-loop edge $e$ of $\Gamma$ and  let $u_e, v_e$ be its ends. Let us  identify $e$  with the interval
 $[-\ell(e)/2,\ell(e)/2]$ so that  its mid-point, $p_e$, gets identified with $0$.
 Now consider the   rational function $f_e$ on $\Gamma$ defined as follows.
  $$
f_e(x)=
\begin{cases}
\begin{array}{ll}
\ell(e)/2+x\ & \text{ if } x\in [-\ell(e)/2,0]  \\
\ell(e)/2-x\ & \text{ if } x\in [0,  \ell(e)/2]  \\
0 &  \text{otherwise, i.e. if } x\not\in e.
\end{array}
\end{cases}
$$
Then
$$
\dv(f_e)=u_e+v_e-2p_e.
$$
If $e$ is a loop based at a vertex $v_e\in V$,  we identify $e$ with the circle obtained from the interval $[-\ell(e)/2,\ell(e)/2]$  by identifying its ends.
Then $f_e$ defines again a rational function, and we have 
$ 
\dv(f_e)=2v_e-2p_e.
$ 
\end{example}
From the previous example we highlight the following basic relation:
\begin{equation}
\label{f_e} 
 u_e+v_e \sim 2p_e 
\end{equation}
  for any edge $e$ of $\Gamma$ with (possibly equal) ends $u_e,v_e$.

Let $O=(\sigma,\tau)$ be a sub-orientation on $G$ supported on the    subdivision $\wh G$. 
For $p\in \Gamma$ we denote by $\deg_{O}^{-}(p)$ the number of edges directed towards   $p$, i.e.
\[
\deg_{O}^{-}(p)=
\begin{cases}
\begin{array}{ll}
|\{e\in E(\wh G) : \tau(e)=p\}|, & \text{ if } p\in V(\wh G);\\
1, &  \text{otherwise}.
\end{array}
\end{cases}
\]
We define the following divisor on $\Gamma$:
\[
D_{O}^-:=\sum_{p\in \Gamma}(\deg^-_O(p)-1+w(p))p.
\]
The divisor $D_{O}^-$ has degree $g-1$ and is supported on some subset of $V(\wh G)$.

\section{Theta-characteristics on tropical curves}
\subsection{Tropical square roots and theta-characteristics}
 The canonical divisor of a  tropical curve $\Gamma=(G,w,\ell)$    is 
\[
K_\Gamma:=\sum_{p\in \Gamma} (2w (p)-2+\deg (p)) p =\sum_{v\in V} (2w (v)-2+\deg (v)) v 
\]
The divisor $K_\Gamma$ has degree $2g-2$, with  $g$ the genus of $\Gamma$.

\begin{defi}
A \emph{(tropical) square  root (of zero)} of $\Gamma$ is a divisor class $[D]\in \Pic(\Gamma)$ such that $[2D]=0$.
The set of square roots   of $\Gamma$ is denoted  by $\RtG$.

A \emph{(tropical) theta-characteristic} of $\Gamma$  is a divisor class $[D]\in \Pic(\Gamma)$ such that $[2D]=[K_\Gamma]$. 
The set of  theta-characteristics    of $\Gamma$ is denoted  by $\TtG$.
\end{defi}
Obviously, $\RtG\subset \Pic^0(\Gamma)$ and $\TtG\subset \Pic^{g-1}(\Gamma)$.

Our next goal is to give an explicit description of all square roots and all theta-characteristics.

Let  $P\in \cC_G$, then $\deg_P(v)$ is even for every $v\in V(G)$. We can thus define the following  divisor on $\Gamma$
\begin{equation}
 \label{F_P}
F_P:=\sum_{v\in V}\frac{\deg_P(v)}{2}v-\sum_{e\in P} p_e. 
\end{equation}
We have, using Example~\ref{fe},
$$
2F_P=\sum_{v\in V} \deg_P(v) v-2\sum_{e\in P} p_e = \sum_{e\in P} (v_e+u_e-2p_e)=\sum_{e\in P}  \dv(f_e).
$$
Therefore $2F_P$ is a principal divisor, in other words, $[F_P]$ is a square root (of zero).
Let us show that, as $P$ varies in $\cC_G$, 
the classes of the divisors $F_P$ are distinct  and give all the square roots of $\Gamma$.

\begin{Prop}
\label{distinct} Let $\Gamma=(G,w,\ell)$ be a tropical curve. Then 
the following is a bijection
$$
\cC_G\la \RtG;\quad \quad P\mapsto [F_P]
$$
\end{Prop}

\begin{proof}
 Recall from \cite{MZ} and \cite{BMV} that there is a natural group isomorphism, $$\Pic^0(\Gamma)\cong H_1(G,\RR) /H_1(G,\ZZ),$$ 
  and the set of 2-torsion points of $\Pic^0(\Gamma)$
(the square roots of $\Gamma$) is identified with $H_1(G,\ZZ/2\ZZ)=\cC_G$. Hence $ \RtG$ has the same cardinality as $\cC_G$.
There remains to prove that the $[F_P]$ are all distinct. 
 It is easy to check that  for every $P,P'\in\cC_G$ we have
\[ 
F_P+F_{P'}-F_{P+P'}=\sum_{e\in P\cap P'}\dv(f_e). 
\]
Hence  
  \[
F_P-F_{P'}\sim F_P-F_{P'} +2F_{P'} = F_P+F_{P'}\sim F_{P+P'}.
\]
Therefore it suffices to show  that $[F_P]\neq 0$  for every  nonzero  $P\in \cC_G$.  By \eqref{F_P}  we have $F_P(p)=0$ unless $p$ is either a vertex of $P$ (in which case $F_P(p)\geq 1$)
 or the mid-point $p_e$ of an edge of $P$ (in which case $F_P(p_e)=-1$).
 By contradiction, assume that for some (nonzero) cycle   $P$ we have $\dv(f)=F_P$, 
with $f$   a rational function on $\Gamma$. Let $\Gamma_{\max}$ be the locus of $\Gamma$ where $f$ attains its maximum. Since  $\Gamma$  is compact, $\Gamma_{\max}$ is a non-empty closed subcurve, moreover   $ \Gamma_{\max}\neq \Gamma$ (as otherwise $\dv(f)=0$).  
 
  We know that the slopes of $f$  outgoing  $\Gamma_{\max}$  are all negative, while  the slopes of $f$
   internal to $\Gamma_{\max}$ are zero
    ($f$ is   constant on  $\Gamma_{\max}$).
 Therefore
   $\dv(f)(p)<0$ for every $p\in \Gamma_{\max}\cap\ol{\ \Gamma\ssm \Gamma_{\max}}$. Since the only points of $\Gamma$ where $F_P$ is negative are the mid-points $p_e$ of the edges in $P$, we deduce  
\begin{equation}
 \label{Gammamax}
 \Gamma_{\max}\cap \ol{\Gamma\ssm \Gamma_{\max}}\subset \{p_e,\  e\in P\}.
 \end{equation}
Let us show that no vertex of $P$ can be contained in  $ \Gamma_{\max}$.  If $v\in V(P)$ were in $ \Gamma_{\max}$,
 it would have to lie in its interior, by \eqref{Gammamax}, and hence  $\dv(f)(v)=0$
 (as $f$ is constant on $\Gamma_{\max}$), which contradicts $F_P(v)\geq 1$.
  
We then conclude that $\Gamma_{\max}$ is   a union of mid-points $p_e$, for some edges $e$ of $P$.
Since the slopes outgoing $\Gamma_{\max}$ are  negative we obtain $\dv(f)(p_e)\leq -2$ for every such $p_e$, 
in contradiction with  $F_P(p_e)=-1$.
\end{proof}

We now turn to theta-characterstics; again, let
 $P\in \cC_G$. Consider the following  divisor on 
$\Gamma$,
\begin{equation}\label{TP}
T_{P}:=\sum_{v\in V}\left(\frac{\deg_P(v)}{2}-1 +w (v)\right)v+\sum_{e\in E \ssm P} p_e.
\end{equation}
Let us show that as $P$ varies the $[T_P]$  give all the theta-characteristics of $\Gamma$.
\begin{Prop}
\label{distinctt} Let $\Gamma=(G,w,\ell)$ be a tropical curve and $P\in \cC_G$. Then $[T_P]$ is a theta-characteristic and
the following is a bijection
$$
\cC_G\stackrel{\beta}{\la} \TtG;\quad \quad P\mapsto [T_P].
$$
\end{Prop}

\begin{proof}
Recalling the definition of $F_P$ we have
\begin{equation}
\label{T_P}
T_{P}=\sum_{v\in V}\left( w(v)-1\right)v+F_P+\sum_{e\in E} p_e.
\end{equation}
As $2F_P\sim 0$ we get
  
$$
2T_{P}\sim \sum_{v\in V}\left(2w (v) 
- 2\right)v + 2\sum_{e\in E } p_e\sim  \sum_{v\in V}\left(2w(v) 
- 2\right)v + \sum_{e\in E }(u_e+v_e)
$$
using \eqref{f_e}. Since $\sum_{e\in E }(u_e+v_e)=\sum_{v\in V}\deg_{G}(v)v$ we have
$2T_P\sim K_{\Gamma}$, i.e. $[T_P]$ is a theta-characteristic.

It remains to show that if $P\neq P'$ then $[T_P]\neq [T_{P'}]$.
Using \eqref{T_P} we have
$$
T_P-T_{P'}=F_P-F_{P'}
$$
and $F_P-F_{P'}\not\sim 0$, by Proposition \ref{distinct}. The proof is complete.
\end{proof}

Notice that the following  
$$
K_{\Gamma}/2=\frac{\sum_{v\in V}(2w (v)-2+\deg(v)) v }{2}
$$
is a divisor on $\Gamma$ if and only if $G$ has only vertices of even degree, i.e. if and only if $G$ is cyclic.
Now, for any $P\in \cC_G$ we denote $\oP=P\cup V$ and
$$
\Gamma_{\oP}=(\oP, w_{|\oP}, \ell_{|P})\subset \Gamma,
$$
in other words $\Gamma_{\oP}$ is the subcurve of $\Gamma$ whose underlying graph is $\oP$, with the same weight function as $\Gamma$.
Then  $K_{\Gamma_{\oP}}/2 $ is a divisor (on both $\Gamma$ and $\Gamma_{\oP}$) and we have
$$
T_{P}=K_{\Gamma_{\oP}}/2  + \sum_{e\in E\ssm P} p_e.
$$

\subsection{Effective theta-characteristics via flows}

We say that a theta-characteristic $[T_P]$ on a tropical curve $\Gamma$ is {\it effective}
if its rank is nonnegative, i.e. if  $r_{\Gamma}(T_P)\neq -1$. Equivalently, $[T_P]$ is effective if there exists an effective divisor $E$ on $\Gamma$ such that $E\sim T_P$.
Our next goal is to characterize  effective theta-characteristics, and describe the effective representatives.

Recall the definition 
 $T_{P}=\sum_{v\in V}\left(\frac{\deg_P(v)}{2}-1 +w(v)\right)v+\sum_{e\in E\ssm P} p_e.
$ 
We write $T_P=T_P^+-T_P^-$ with $T_P^+$ and $T_P^-$ effective.
We have
$$T_{P}^-=\sum_{\stackrel{v \in V\ssm V(P)}{w(v)=0} }v.
$$
In other words, $T_P\geq 0$ if and only if every vertex of $G$ having   weight zero is contained in $P$.

\begin{defi}
 Let $P\in\mathcal C_G$   and $W\subset V $.  Assume that $P\cup W$ is not empty.
The \textit{cyclic subcurve} of $\Gamma$ associated to $P$ and $W$ is the tropical subcurve $\Gamma_{P,W}\subset \Gamma$ supported on the graph $P\cup W$ (we consider $W$ to be a subgraph of $G$ with vertex set $W$ and no edges).
\end{defi}

Consider a  cyclic subcurve $\Gamma_{P,W}$ of $\Gamma$. Let us define a sub-orientation $O_{P,W}$ on $G$ and a divisor $D_{P,W}$ on $\Gamma$. Informally,   $O_{P,W}$ will be the {\it flow} away from  $\Gamma_{P,W}$.
More precisely, we choose on
 $P$   a cyclic orientation; 
there are   $2^{b_1(P)}$ ways to do it, the choice of which will be irrelevant.
Next, we orient 
the edges in the complement of $P$   away from $\Gamma_{P,W}$, as follows.
Consider the distance function
$$\begin{aligned}d_{P,W}\col& \Gamma \longrightarrow \mathbb R_{\geq 0}\\
&p\longmapsto d(p,\Gamma_{P,W}),
\end{aligned}
$$
where $d(p,\Gamma_{P,W})$ denotes the distance from $p$ to $\Gamma_{P,W}$. 
The function $d _{P,W}$ is easily seen to be continuous and  piecewise linear with slope 0 on $P$ and $\pm1$ away from $P$.
 Moreover, for any edge $e$ of $G$, the function $d _{P,W}$  has at most one critical point lying in the interior
 of $e$.
 
 We consider the subdivision $\wh G_{P,W}$ of $G$ given by inserting a vertex  at each critical point    of $d _{P,W}$. We orient   each edge of $\wh G_{P,W}-P$ towards the maximum of  $d_{P,W}$, so that we have  a  sub-orientation $O_{P,W}$ on $G$. 
We set
\begin{equation}\label{eq:DPW}
D_{P,W}:=D^-_{O_{P,W}}=\sum_{v\in V(\wh G_{P,W})}\left(\deg^-_{O_{P,W}}(v)-1+w_\Gamma(v)\right)v.
\end{equation}
 \begin{remark}
 If $W=V$ every vertex $v$  not lying on $P$ is a source 
 (i.e. $\deg^-_{O_{P,V}}(v)=0$)
 and $p_e$ is a sink (hence $\deg^-_{O_{P,V}}(p_e)=2$)
 for
 every edge $e $ not in $P$.  Therefore
by \eqref{TP}  we have
\begin{equation}\label{eq:DP}
D_{P,V}=
T_P.
\end{equation}
\end{remark}

\begin{prop}\label{DPS}
Let  $\Gamma_{P,W}$ and $\Gamma_{P',W'}$ be two cyclic subcurves  of a tropical curve  $\Gamma$. Then
\begin{enumerate}
\item
\label{DPS1}
 $D_{P,W}\sim D_{P',W'}$ if and only if $P=P'$. 
 \item
 \label{DPS2}
$[D_{P,W}]$ is a theta-characteristic on $\Gamma$, that is, $2D_{P,W}\sim K_{\Gamma}$.
\end{enumerate}
\end{prop}

\begin{proof}
We first prove   $D_{P,W}\sim  D_{P,W\cup\{v\}}$ for any vertex $v\in \Gamma$.
Consider the following  function $f$ on $\Gamma$ 
$$f:=\frac{d_{P,W\cup\{v\}} -d_{P,W}}{2}.$$
Of course,  $f$ is continuous and piecewise linear  with finitely many pieces. We will show that $f$ has integer slopes, i.e.  is a rational function, and that $D_{P,W}-D_{P,W\cup \{v\}}=\dv(f)$.

Set $O_1:=O_{P,W}$ and $O_2:=O_{P,W\cup\{v\}}$. Assume without loss of generality that $O_1$ and $O_2$ coincide  over $P$. For  $p\in\Gamma$ we have
\begin{equation}\label{eq:difference}
D_{P,W}(p)-D_{P,W\cup \{v\}}(p)=\deg^-_{O_1}(p)-\deg^-_{O_2}(p).
\end{equation}
Let $B(p,\epsilon)\subset \Gamma$ be the closed ball with center $p$ and radius $\epsilon\in \mathbb R_{>0}$. If $\epsilon$ is small enough   $B(p,\epsilon)$ is the union of segments $h_i$ of length $\epsilon$ and  incident to $p$, for $i=1,\dots,\deg_{\Gamma}(p)$, over which both $d_{P,W}$ and $d_{P,W\cup\{v\}}$ are linear. Fix one such  $h_i$. 

If   $O_1$ and $O_2$ coincide along $h_i$, then $d_{P,W}$ and $d_{P,W\cup\{v\}}$ have the same slope on it. Therefore, their difference is a constant function on $h_i$, and the   slope of $f$ at $p$ along $h_i$ is $0$.  

If  $O_1$ and $O_2$  do not coincide along $h_i$, then $h_i$ is not contained in $P$, the slopes of $d_{P,W}$ and $d_{P,W\cup\{v\}}$  have absolute value $1$ and opposite signs. Therefore their sum is even  and  $f$ has slope $\pm 1$ along $h_i$.
So $f$ is a rational function. Suppose that $p$ is the target of $h_i$ according to $O_1$, and the source according to $O_2$,
Then  the contribution to $\deg^-_{O_1}(p)$ along $h_i$ is $1$, and the contribution to $\deg^-_{O_2}(p)$ is zero, so their difference is $1$.
Moreover, the slope of $d_{P,W}$ on $h_i$ is equal to $-1$
whereas the slope of $d_{P,W\cup \{v\}}$ is $+1$, hence the contribution to $\dv(f)(p)$ along $h_i$ is $(1-(-1))/2=1$.

Repeating this  for all segments $h_i$ we get
 $\dv(f)(p)=\deg^-_{O_1}(p)-\deg^-_{O_2}(p)$, hence by \eqref{eq:difference} we have $\dv (f)(p)=D_{P,W}(p)-D_{P,W\cup \{v\}}(p)$, as wanted.

Therefore  $D_{P,W}\sim  D_{P,W\cup\{v\}}$, hence  $D_{P,W}\sim D_{P, W'}$ for any $W,W'$. 

In particular $D_{P,W}\sim D_{P,V}$, hence, as
  $D_{P,V}=T_V$, by Proposition~\ref{distinctt} $D_{P,W}$ is a theta-characteristic.  Part \eqref{DPS2} is proved.
  
It remains to prove that $D_{P,W}$ and $D_{P',W'}$ are not equivalent if $P\ne P'$. 
Now it is enough to show that $D_{P,V}$ and $D_{P',V}$ are not equivalent, i.e. that $T_P$ and $T_{P'}$ are not equivalent, which follows again by Proposition \ref{distinctt}.
 \end{proof}

We write  $V_+(\Gamma)=\{v\in V:\  w(v)>0\}$, sometime  just  $V_+=V_+(\Gamma )$. We say that $\Gamma$ is a {\it pure} tropical curve if $V_+(\Gamma)=\emptyset$, equivalently, if $w=0$.

We   now  prove that     theta-characteristics are always effective with the only exception of
the  theta-characteristic $[T_0]$ on a pure tropical curve.
 \begin{Thm}\label{thm:rank}
Let $\Gamma=(G,w,\ell)$ be a tropical curve and
$P\in \cC_G$.
 Then $r_{\Gamma}(T_P)=-1$ if and only if $P=0$ and    $\Gamma$ is   pure.
\end{Thm}
The case of a pure tropical curve is known; see \cite{zharkov}. We include the full proof for completeness and better clarity.
First  we note the following simple fact.
\begin{lemma}
\label{eff}
Let $P\in \cC_G$ and $W\subset V$ with $P\cup W$ non empty. Then
\begin{enumerate}
\item
\label{eff1}
$D_{P,W}\geq 0$ if and only if $w(v)>0$ for every $v\in W\ssm V(P)$.
 \item
  \label{effz}
$D_{0,W}\geq 0$ if and only if  $W\subset V_{+}$.
\item
  \label{effe}
$D_{P,\emptyset }\geq 0$.
\end{enumerate}

\end{lemma}
 
\begin{proof}
Notice that 
 \eqref{effz} and  \eqref{effe}  follow  trivially  from \eqref{eff1}.

Let us  prove \eqref{eff1};  we have
 $$
\deg^-_{O_{P,W}}(v) 
\begin{cases}
\begin{array}{ll}
= 0  &   \text{ if }   v  \in W\ssm V(P)\\
 \geq 1 &   \text{ otherwise. }
\end{array}
\end{cases}
$$
 Therefore, writing $D_{P,W}=D^+_{P,W}-D^-_{P,W}$ as the difference of two effective divisors, we have
 $$
 D^-_{P,W}=\sum _{\stackrel{v \in  W\ssm V(P) }{w(v)=0} }v
 $$
 which is zero if and only if $w(v)>0$ for every $v \in  W\ssm V(P)$. 
\end{proof}
Now we prove Theorem~\ref{thm:rank}. 
\begin{proof}
By Lemma~\ref{eff}\eqref{effe}, if $P\neq 0$ then $D_{P,\emptyset}\geq 0$.
By Proposition~\ref{DPS} we have $D_{P,\emptyset}\sim D_{P,V}$ and since 
  $D_{P,V}= T_P$ we conclude that $[T_P]$ is effective.

We are left with the case $P=0$. 

If $\Gamma$ is not pure then  $V_+\neq \emptyset$, hence we can consider
$D_{0,V_+}$, which is a representative for $[T_0]$.  By Lemma~\ref{eff} we have $D_{0,V_+}\geq 0$, hence $[T_0]$ is effective.
This concludes the proof in case $\Gamma$ is not pure.

 Now suppose $\Gamma$ is   pure.    To finish the proof we need to show  
 $$r_{\Gamma}(T_0)=-1$$
It suffices to prove that $r_{\Gamma}(D_{0,v})=-1$ for any vertex $v$ of $\Gamma$.
By definition, $D_{0,v}$ is associated to the  sub-orientation  $O_{0,v}$, which is clearly  acyclic
(having $v$ as a source).  Therefore, $D_{0,v}$ is a reduced divisor with respect to $v$ in the sense of \cite{BNRR}, so the fact that $D_{0,v}(v)<0$ implies that it has rank equal $-1$ (see e.g., Lemma 3.8(a) in \cite{CLM}). In fact, $D_{0,v}$ is a moderator in the sense of \cite[Def. 7.8]{MZ}, so the fact that it has rank $-1$ follows also from
 \cite[Lm. 7.10]{MZ}. \end{proof}

\begin{remark}
  If $\Gamma$ is a pure tropical curve, then the theta-characteristics described by Zharkov in \cite{zharkov} are the classes of the divisors $D_{P,W}$ where $W=\emptyset$ for $P\ne 0$, and $|W|=1$ for $P=0$.
\end{remark}

 \section{The moduli space of tropical theta-characteristics}
The moduli space, $\Mgt$, of tropical curves of genus $g$, first  constructed in \cite{BMV},
has the structure of a  generalized cone complex   (see   \cite{ACP}). 
The points in $\Mgt$ are in bijective correspondence with (equivalence classes of) tropical curves of genus $g$.

Similarly, 
  the moduli space $\Sgt$, of   spin tropical curves of genus $g$, constructed in \cite[sec. 2.5]{CMP1} is a generalized cone complex, and there is a natural morphism of generalized cone complexes
  $$
  \pi^{\trop} :\Sgt\la\Mgt.
  $$
This morphism has an explicit connection with the analogous  moduli spaces of Deligne-Mumford stable algebraic curves, $\Mgbst$, and stable spin curves, $\Sgbst$, for which we refer to loc. cit. Theorem C.

We shall now  construct the moduli space of tropical theta-characteristics on tropical curves of genus $g$ as 
 a generalized cone complex, and relate it to $\Mgt$ and $\Sgt$. 
  Since the procedure is very similar to the one used in \cite{ACP} and \cite {CMP1} we will skip
many details.

By Proposition~\ref{distinctt}, for any  tropical curve   
$\Gamma=(G,w, \ell)$ of combinatorial type $(G,w)$, we have the  isomorphism
  $ 
\cC_G\cong \TtG
$  
 associating to each cyclic subgraph $P$ of $G$ the theta-characteristic $[T_P]$. Therefore, for each tropical curve of combinatorial type $(G,w)$, there is exactly one theta-characteristic associated to the choice of each cyclic subgraph $P$ of $G$.

Notice that a non-trivial automorphism of $\Gamma$ may fix a cyclic subgraph of $G$, that is, $\Aut(\Gamma)$ may act non-trivially  
 on $\TtG$. 
 
 
 For any $P\in \cC_G$ we denote by
  $T^{\trop}_{(G,P)}$
 the set of isomorphism classes of all theta-characteristics of type $[T_P]$ on all tropical curves $\TC$.

 We consider
the poset of cyclic graphs of genus $g$, i.e.:
$$
 \cC_g:=\bigsqcup_{G\in \Sg }\cC_G
$$
where $\Sg$ is the poset of stable graphs of genus $g$, partially ordered by edge contraction. 
Elements in $ \cC_g$ are written as pairs $(G,P)$ with $G$ a stable graph and $P\in \cC_G$.
The poset structure on $ \cC_g$ is given by edge contraction, as follows. For
$(G,P)$ and $(G',P')$ in $\cC_g$ 
we 
say that $(G,P)\geq (G',P')$ 
 if there exists a contraction $\gamma\col G\to G'$ such that $\gamma_*P=P'$ (in this case we have $G\geq G'$ by definition). By  \cite[Prop. 2.3.1]{CMP1}
 the poset $ \cC_g$ is connected and the forgetful map $ \cC_g\to \Sg$ is a quotient of posets. We also consider the poset $[\cC_g]:=\sqcup_{G\in \Sg} \cC_G/\Aut(G)$.

We introduce the category, {\sc{cyc}}$_{g}$, whose objects are isomorphism classes of pairs
$(G,P)$ with $G\in \Sg$ and $P\in \cC_G$,
 and whose arrows are generated by contractions and automorphisms of pairs (that is, the elements of the subgroup $\Aut(G,P)$ of $\Aut(G)$ fixing $P$). 
 
 We now  define a contravariant functor from the category {\sc{cyc}}$_{g}$   to the category of rational polyhedral cones.
 To the isomorphism class of   $(G,P)$ we associate the cone
$$
\sigma_{_{(G,P)}}=\RR_{\geq 0}^{E(G)},
$$
with  the   integral structure determined by the sub-lattice parametrizing tropical curves having  integral edge-lengths.
As usual, we write $\sigma_{_{(G,P)}}^o=\RR_{> 0}^{E(G)}$.

To a contraction  $\gamma\col (G,P)\to (G',P')$ we associate the injection of cones 
$ 
\iota_{\gamma}\col \sigma_{_{(G',P' )}}\ha \sigma_{_{(G,P )}}
$ 
whose image is the face of $\sigma_{_{(G,P)}}$ where the coordinates corresponding to $E(G)\smallsetminus E(G')$ vanish.
If $\gamma\in \Aut(G,P)$   then $\iota_{\gamma}$ is the corresponding   automorphism of $\RR_{\geq 0}^{E(G)}$.
By the results in \cite[sect. 2]{CMP1}, this is indeed a contravariant functor.
We can therefore consider  the colimit of the diagram of cones $\sigma_{_{(G,P)}}$ using the inclusions $\iota_{\gamma}$,
for all arrows, $\gamma$, in {\sc{cyc}}$_{g}$, which is the following generalized cone complex
  $$
\Ttg:=  \varinjlim \left(\sigma_{_{(G,P)}}, \iota_{\gamma}\right).
 $$

\begin{Thm}
\label{Tgthm} 
The following properties hold.
\begin{enumerate}
 \item
 \label{Tgthm1} 
The space $\Ttg$ is 
the moduli space of   tropical theta-characteristics and we have a stratification
 $$
\Ttg=\bigsqcup_{[G,P]\in [\cC_g]} T^{\trop}_{(G,P)}. 
$$
\item
We have
$
  T^{\trop}_{(G,P)}\cong\sigma_{_{(G,P)}}^o/ \Aut(G,P).
$

 \item $\Ttg$ is connected and  has    pure dimension $3g-3$.

   \item $ T^{\trop}_{(G',P')}\subset \overline{T^{\trop}_{(G,P)}}$ if and only if $(G,P)\geq (G',P')$.
\end{enumerate}
\end{Thm}
\begin{proof}
A  point  $\underline{l}=(l_e, \  e\in E) \in \sigma_{_{(G,P )}}$ corresponds to a pair  $(\Gamma_{\underline{l}}, P)$ where 
$\Gamma_{\underline{l}} = (G, w, \ell,)$ is the tropical curve 
whose length function is $\ell(e)=l_e$ for all $e\in E$.  Hence $\underline{l}$ corresponds to the theta-characteristic
$T_P\in T^{\trop}_{\Gamma_{\underline{l}}}$. 
By extending this reasoning to all tropical curves of genus $g$    we have a surjection
$$
\bigcup _{\Gamma \in \Mgt}\TtG\la \Ttg 
$$
which identifies $(\Gamma, T_P)$ with $(\Gamma', T_{P'})$ if and only if $\Gamma = \Gamma'$ and there is an automorphism of $\Gamma $ mapping $T_P$ to $T_{P'}$. From this \eqref{Tgthm1} follows. 

The rest  is   similar to the proofs of    \cite[Props   2.5.1 and  2.5.2]{CMP1}.
\end{proof}

We have a canonical  
morphism of generalized cone complexes
\[
\psi^{\trop}\col   \Ttg \la  \Mgt 
\]
sending a tropical theta-characteristic  $ [T_P] $ in
$\Gamma=(G,w,\ell)$ (associated to the cyclic subgraph $P$ of $G$)
  to the     tropical curve $ (G,w, \psi^{\trop}(\ell)) $, where 
$
\psi^{\trop} (\ell)(e) =2\ell(e)$ if $e\in E\smallsetminus P$, and $\psi^{\trop} (\ell)(e) = \ell(e)$ if $e \in P$.

By the same argument as in \cite[Prop. 2.5.3]{CMP1} we have, for every  tropical curve $[\Gamma]\in \Mgt$,
\[
{ (\psi^{\trop})}^{-1}( \Gamma )\cong \TtG/\Aut(\Gamma).
\]

 \begin{remark}\label{prop:pitrop} The canonical map $\pi^{\trop} :\Sgt\la\Mgt$  factors through $\psi^{\trop}$. 
 \end{remark}

\section{Lifting theta-characteristics on tropical curves}

From now on, we fix an algebraically closed field $k$  of characteristic different from $2$.

Let $X$ be a stable curve over $k$.
We let $(G,w)$ be the  (stable) dual   graph of $X$, with $G=(V,E)$. 
Recall that $V$ corresponds to the set of irreducible components of $X$ and $E$  to the set of its nodes. We often identify edges with nodes. The weight of a vertex is the genus of the desingularization of the corresponding component, therefore  $(G,w)$ is a stable graph and has   the same genus as $X$.

Recall that a \emph{spin curve} over $X$ is a pair $(\wh X, \wh L)$, where $\wh X=X^\nu_R\cup Z$ is a quasistable curve, with $X^\nu_R$ the normalization of $X$ at the subset of nodes corresponding to some $R\subset E$ and $Z$ the disjoint union of smooth connected rational components meeting $X^\nu_R$ at $2$ points (called \emph{exceptional components}), and $\wh L$ is a line bundle on $\wh X$ such that
\begin{enumerate}[(1)]
 \item 
the restriction, $L_R$, of  $\wh L$ to $X^{\nu}_R$ satisfies $L_R^2\cong \omega_{X^{\nu}_R}$;
 \item
 the restriction of $\wh L$  to each exceptional component $E$  has degree $1$.
   \end{enumerate}
   
We denote by  ${\cS}_X$ the scheme of spin curves over $X$. 
We will use the terminology \emph{theta-characteristic} for spin curves such that $\wh X=X$ (in this case, $\wh L^2\cong \omega_X$).
   The \emph{parity} of a spin curve refers to the parity of $h^0(\wh X, \wh L)$.

 The {\it dual spin graph} of a spin curve $(\wh X, \wh L)$ on $X$
 is the tern $(G,P ,s )$, where:
\begin{itemize} 
\item[(1)]  $G$ is the dual graph of $X$;
\item[(2)] $P=E \smallsetminus R$ (it is a cyclic subgraph of $G$);
\item[(3)] $s\col V(G/P)\ra \mathbb Z/2\mathbb Z$ is a function taking $v\in V(G/P)$ to $s(v)$, the parity of $h^0(Z_v, \wh L_{|Y_v})$, where $Y_v$ is the connected component of $X^\nu_R$ corresponding to $v$ (the set of connected components of $X^\nu_R$ is in bijection with the set of vertices $V(G/P)$). 
\end{itemize}
Recall that $(G,P,s)$ is a spin graph in the sense of \cite[Def. 2.1.1]{CMP1} and its parity is the parity of $\sum_{v\in V(G/P)} s(v)$.

 \

In what follows, we fix an algebraically closed non-Archimedean field  $K$   whose valuation ring is  
$R$  and residue field is $k$.

We let 
$X_K$ be a genus-$g$ smooth curve over $K$, we
 assume that  $X_K$ extends to a stable curve $\cX$ over $\Spec R$, whose special fiber we denote
by  $X$.

 We let  $\Gamma=(G,w,\ell)$  be the tropical curve given as    the skeleton of the Berkovich analitification $X_K^{\an}$. Recall that the length   $\ell$ is defined by setting $\ell(e)=\val_K(f_e)$ 
  where $xy=f_e$ is an \'etale local equation of $\cX$ at the node corresponding to the edge $e$.

The natural retraction map $\tau\col X_K^{\an}\ra \Gamma$    induces by linearity a specialization homomorphism: 
\[
\tD\col {\Div}(X_K)\la \Div(\Gamma)
\]
 where $\Div(X_K)$ is the group of divisors on $X_K$. 

Let us describe $\tD$ explicitly.
  Let $D_K$ be a 
  divisor on $X_K$.   Then there is a semistable curve $\wh{\cX}$ over $\Spec R$ with $\cX$ as stable model, such that $D_K$ extends to a Cartier divisor $\cD$ over $\wh{\cX}$; we let $\wh X$ be the special fiber of  $\wh{\cX}$.  The dual graph $\wh G$ of $\wh X$ is obtained by inserting $n_e$ vertices of weight zero in the interior of any edge $e$ of $G$, where $n_e$ is the number of rational components of $\wh X$ lying over the node of $X$ corresponding to $e$. 
Hence we  can  view $V(\wh G)$ as a subset of $\Gamma$ (see also Remark~\ref{addpt}).
The multidegree, $\mdeg\cD|_{\wh X}$, of the restriction of $\cD$ to ${\wh X}$ is thus a divisor of $\Gamma$ supported at $V(\wh G)$, and we have
\begin{equation}\label{eq:tau*}
\tD(D_K)=\mdeg\cD|_{\wh X}.
\end{equation}

\begin{remark}
\label{addpt}
 In passing from $\cX$ to  $\wh{\cX}$, the tropical equivalence class of the skeleton $\Gamma$ does not change.
Indeed, this amounts to  inserting weight zero vertices in the interior of the edges of $\Gamma$, getting a tropical curve, $\wh{\Gamma}=(\wh G,\wh{w},\wh{\ell})$,  having the same metric structure, and hence
  equivalence class, of $\Gamma$. More precisely (following \cite{ACP} Section 8)
  let $v$ be the vertex of $\wh G$ corresponding to the exceptional component, $C_v$, of $\wh X$.
  Let $e_0$ and $e_1$ be the edges of $\wh{\Gamma}$ corresponding to the two nodes
  in $C_v\cap\ov{\wh X \ssm C_v}$, and, for $i=0,1$, let $xy=f_{e_i}$ be an \'etale local equation of $\wh{\cX}$ at $e_i$. Then an \'etale local equation of $\cX$ at the node $e$ to which $C_v$ gets contracted is $xy=f_{e_0}f_{e_1}$,   therefore 
  $$
  \ell(e) =\val_K(f_{e_0}f_{e_1})=\val_K(f_{e_0})+\val_K(f_{e_1})=\hat{\ell}(e_0)+\hat{\ell}(e_1).
  $$ 
 \end{remark}

The map $\tD$ takes principal divisors on $X_K$ to principal divisors on $\Gamma$ (see  \cite{BR15}),
hence we get a homomorphism 
\[
\tP\col \Pic(X_K)\longrightarrow \Pic(\Gamma).
\] 

Let $\cSK\subset \Pic(X_K)$  and $\TtG\subset \Pic(\Gamma)$ be, respectively, the sets of theta-characteristics of $X_K$ and of $\Gamma$. Since $\tP$ is a homomorphism taking class of the canonical divisor of $X_K$ to the class of the canonical divisor of $\Gamma$, it restricts to a map: 
\[
\tS\col \cSK \longrightarrow \TtG.
\]

In Proposition \ref{distinctt} we introduced the bijection  $\beta: \cC_G\longrightarrow \TtG$ mapping $P$ to $[T_P]$. We now use it to construct  a useful factorization of the map  $\tS$. First of all, 
we define a map  
\[
\nS \col  \cSK \longrightarrow\cC_G
 \]
as follows. Given a theta-characteristic $L_K$ on $X_K$,    there is a  
unique pair $(\wh{\cX}, \wh{\mc L})$ with the following properties.
  $\wh{\cX}$ is a semistable curve over $\Spec R$ having $\cX$ as stable model,   $\wh{\mc L}$ is a line bundle on   $\wh{\cX}$ extending $L_K$;  we let  $\wh X$ be the special fiber of $\wh{\cX}$
and $\wh L=\wh{\mc L}_{|\wh X}$. Finally (which ensures uniqueness) $(\wh X,\wh L)$  is a spin curve on $X$.  We denote by
  $(G,P,s)$  the  dual spin graph  of $(\wh X,\wh L)$. 
  By definition,  $P\in \cC_G$, and we define 
\[
\nS (L_K)=  P. 
\]

\begin{Lem}\label{lem:factor}
With the above notation, we have the following factorization 
\[
\tS \col \cSK  \stackrel{\nS}{\la} \cC_G\stackrel{\beta}{\la} \TtG.
\]
\end{Lem}

\begin{proof}  We continue to use the   notation before the statement. 
  We need to show that $\tS (L_K)$ is  equal to   $[T_P]$.
By \eqref{eq:tau*}
$$
\tS (L_K)=\mdeg \wh{\mc L}_{|\wh X} = \mdeg \widehat{L}.
$$

 To compute the multidegree of $\widehat{L}$, observe that
 the dual graph, $\wh G$, of $\wh X$ is  
 the $R$-subdivision of $G$, where $R=E\setminus P$.
  For a vertex $v\in V(\wh G)$, we let $\wh C_v$ be the irreducible component of $\wh X$ corresponding to $v$.  Recall that $\widehat{L}$ restricts to a theta-characteristic on the complement of the exceptional components of $\wh X$, i.e. on the subcurve of $\wh X$ whose dual graph is $(V(G),P)$.
  
  Therefore  $ \forall v\in V(G) \subset V(\wh G)$ we have
\[
\deg_{\wh C_v} \widehat{L}=w(v)-1+\frac{\deg_P(v)}{2}.
\]
Comparing this with   \eqref{TP} we conclude that the coefficient of $v$ in  $\mdeg \widehat{L}$ and  $T_P$
coincide, as wanted.

Consider now $v\in V(\wh G)\ssm V(G)$, so that $\wh C_v$ is an exceptional component lying over a node of $X$ corresponding to an edge $e\in E\ssm P$.
We know that $\deg_{\wh C_v}\widehat{L}=1$. To conclude the proof, in view of \eqref{TP}, we need to show that in the tropical curve $\Gamma$ the vertex $v$ coincides with $p_e$, the mid-point of $e$.

We use Remark~\ref{addpt}; let $e_0$ and $e_1$ be the nodes of $\wh{X}$ lying on $\wh C_v$.
 Recall that there are \'etale neighborhoods of $\wh{\cX}$ around $e_0$ and $e_1$ in which the local equations of  $\wh{\cX}$ are, respectively, $x_0y_0=h_e$ and $x_1y_1=h_e$, for the same $h_e\in R$ (see \cite[Eq. (4), Section 3.2]{CCC07}). This implies   that the lengths of $e_0$ and $e_1$ are equal, and therefore $v$ is the mid point of $e$.
\end{proof}

We let $\cSK^+$ and $\cSK^-$, respectively, be the loci in $\cSK$ corresponding to even and odd theta-characteristics of $X_K$. Thus $\tS$ restricts to maps
\[
\tS^+ \col \cSK^+ \longrightarrow \TtG
\quad \text{ and } \quad
\tS^- \col \cSK^-\longrightarrow \TtG.
\]
Similarly, we    restrict $\alpha$  
\[
\nS^+\col  \cSK^+\longrightarrow\cC_G
\;\quad \text{ and } \;\quad
\nS^-\col  \cSK^-\longrightarrow\cC_G
\]

\

Next, let $\mc S_X$ be the scheme of spin curves over $X$. For a spin structure $(P,s)$ on $G$ we define
\[
\mc S_{(X,P,s)}=\{(\wh X,\wh L) \in \mc S_X: \text{the dual spin graph of } (\wh X, \wh L) \text{ is } (G,P,s)\}
\]
Denote by $\mc S_X^+$ and $\mc S_X^-$, respectively, the subschemes of $\mc S_X$ corresponding to even and odd spin curves. 

Recall that $SP_G^+$ and $SP_G^-$ denote, respectively, the set of pairs $(P,s)$ such that $(G,P,s)$ is an even and odd spin graph.  
For every $P\in \cC_G$, we define
\[
\mc S^+_{(X,P)}=\cup_{(P,s) \in SP_G^+} \mc S_{(X,P,s)}
\quad \text{ and }\quad
\mc S^-_{(X,P)}=\cup_{(P,s) \in SP_G^-} \mc S_{(X,P,s)}.
\]

\begin{Lem}\label{lem:ramif}
With the above notation, for every $P\in \cC_G$ we have 
  \[
  |(\nS^\pm)^{-1}(P)|=2^{b_1(G)-b_1(P)}|\mc S_{(X,P)}^\pm|.
\]  
\end{Lem}

\begin{proof}
As we saw above, for every theta-characteristic $L_K$ on $X_K$ there is a unique pair $(\wh{\cX},\wh{\mc L})$ over $\Spec R$ extending $(X_K,L_K)$. By construction, the special fiber of  $(\wh{\cX},\wh{\mc L})$ is contained in $\mc S_X$. For every $(\wh X, \wh L)\in \mc S_X$, we let $N(\wh X, \wh L)$ be the number of theta-characteristics $L_K$ on $X_K$ whose extension $(\wh{\mc X},\wh{\mc L})$  has $(\wh X,\wh L)$ as a special fiber. Clearly, we have
\[
|(\nS^\pm)^{-1}(P)|=\sum_{(\wh X, \wh L)\in \mc S_{(X,P)}^\pm} N(\wh X, \wh L).
\]

Let $\pi\col S_{\cX}\ra \Spec R$ be the moduli space of spin curves of the family $\cX\ra \Spec R$. The map $\pi$ is finite and flat and  the multiplicity of the special fiber of $\pi$ at a point $(\wh X, \wh L)\in \mc S_{(X,P)}^\pm$ is  $2^{b_1(G)-b_1(P)}$ (see \cite{C89} and \cite{CC03}). Thus $\wh L$ lifts to  $2^{b_1(G)-b_1(P)}$ theta-characteristics on $X_K$, that is, $N(\wh X, \wh L)=2^{b_1(G)-b_1(P)}$, and we are done.
\end{proof}

Given an integer $m$, we set
\[
 N^+_m:=2^{m-1}(2^m+1) \quad \text{ \; and \; } \quad N^-_m:=2^{m-1}(2^m-1).
\]

\begin{Thm}\label{thm:main2}
For every $P\in \cC_G$, the following properties hold. 
\begin{enumerate}
\item
\label{thm:main2a} If $P\ne0$, then 
\[
|(\tS^+)^{-1}(T_P)|=|(\tS^-)^{-1}(T_P)|=2^{2g-b_1(G)-1}.
\]

\item
\label{thm:main2b} 
If $P=0$, then 
\[
|(\tS^+)^{-1}(T_0)|=2^{b_1(G)}\underset{|U|\equiv 0\mod  (2)}{\sum_{ U\subset V}}\left(\prod_{v\in U} N^-_{w(v)}\prod_{v\in V\setminus U} N^+_{w(v)}\right).
\]
and
 \[
|(\tS^-)^{-1}(T_0)|=2^{b_1(G)}\underset{|U|\equiv 1\mod  (2)}{\sum_{ U\subset V}}\left(\prod_{v\in U} N^-_{w(v)}\prod_{v\in V\setminus U} N^+_{w(v)}\right).
\]
\end{enumerate}
\end{Thm}

\begin{proof}
Consider   $(\wh X,\wh L)\in S_{(X,P)}$. 
The dual graph of $\wh X$ is the $R$-subdivision of $G$, where $R=E\ssm P$.   We let $Z\subset \wh X$ be the union of the exceptional components, and $X_1,\dots,X_n$   the connected components of $\overline{\wh X\setminus Z}$. It is well known (as the restriction of $\wh L$ to every exceptional component is $\cO(1)$) that we have the following identity
\begin{equation}\label{eq:sum}
h^0(\wh X, \wh L)=\sum_{1\le i\le n} h^0(X_i,\wh L|_{X_i}).
\end{equation}

Assume   $P\ne0$.  We can assume that the dual graph $P_1$ of $X_1$ has $b_1(P_1)\ne0$. By  \cite[Cor. 2.13]{Har82}, we have the same number, $M$, of odd and even theta-characteristics on $X_1$.  We argue that $|\mc S_{(X,P)}^+|=|\mc S_{(X,P)}^-|$.  By \eqref{eq:sum}, this is clear if $n=1$.
Otherwise, let $Y$ be the disjoint union of $X_2,\dots,X_n$, and denote by  $A$ and $B$ the number of even and odd theta-characteristics on $Y$, respectively. We have
$h^0(\wh X, \wh L)= h^0(X_1,\wh L|_{X_1})+h^0(Y,\wh L|_{Y})$, therefore
$$|\mc S_{(X,P)}^+|=|\mc S_{(X,P)}^-|=AM+BM.$$
Therefore, by   \cite[Sect. 1.3]{CC03},
\[
|\mc S_{(X,P)}^+|=|\mc S_{(X,P)}^-|=2^{2\sum_{v\in V}w(v)}2^{b_1(P)-1}=2^{2g-2b_1(G)}2^{b_1(P)-1}.
\]
Then  by Lemmas \ref{lem:factor} and \ref{lem:ramif} we have
\[
|(\tS^\pm)^{-1}(T_P)| = 2^{b_1(G)-b_1(P)}|\mc S_{(X,P)}^\pm| =  2^{2g-b_1(G)-1}.
\]
This concludes the proof of \eqref{thm:main2a}.

Assume now $P=0$. A spin curve in $\mc S_{(X,0)}$ is given by the datum of a theta-characteristic  on each component of the normalization of $X$. This spin curve is even if and only if we have an even number of components for which the theta-characteristic is odd. Hence  
\[
|\mc S_{(X,0)}^+|=\underset{|U|\equiv 0\mod  (2)}{\sum_{ U\subset V}}\left(\prod_{v\in U} N^-_{w(v)}\prod_{v\in V\setminus U} N^+_{w(v)}\right)
\]
which, by Lemmas \ref{lem:factor} and \ref{lem:ramif}, gives the stated number for $|(\tS^+)^{-1}(T_0)|$. A similar reasoning gives the stated number for $|(\tS^-)^{-1}(T_0)|$.
\end{proof}

\begin{rem}
Notice that Theorem \ref{thm:main2} recovers \cite[Theorem 1.1, (2) and (3)]{JL}. 
Indeed, assume that $\Gamma$ is pure, so that $b_1(\Gamma)=g$. For every $P\ne0$, by Theorem \ref{thm:main2} we get 
\[
|(\tS^+)^{-1}(T_P)|=|(\tS^-)^{-1}(T_P)|=2^{g-1}.
\]
Moreover, for $\epsilon\in\{0,1\}$, using that $N^+_0=1$ and $N^-_0=0$, we have
\[
\underset{|U|\equiv\epsilon \mod  (2)}{\sum_{ U\subset V}}\left(\prod_{v\in U} N^-_0\prod_{v\in V\setminus U} N^+_0\right)=
\begin{cases}
\begin{array}{ll}
1, & \text{ if } \epsilon=0, \\
0, & \text{ if } \epsilon=1,
\end{array}
\end{cases}
\]
(the unique non-zero summand in the left hand side is the one corresponding to $\epsilon=0$ and $U=\emptyset$). By Theorem \ref{thm:main2},
\[
|(\tS^+)^{-1}(T_0)|=2^g
\quad \text{ and } \quad 
|(\tS^-)^{-1}(T_0)|=0.
\]
\end{rem}

\end{document}